\documentclass[twoside,leqno,10pt, A4]{amsart}
\usepackage{amsfonts}
\usepackage{amsmath}
\usepackage{amscd}
\usepackage{amssymb}
\usepackage{amsthm}
\usepackage{amsrefs}
\usepackage{latexsym}
\usepackage{mathrsfs}
\usepackage{bbm}
\usepackage{amscd}
\usepackage{amssymb}
\usepackage{amsthm}
\usepackage{amsrefs}
\usepackage{latexsym}
\usepackage{mathrsfs}
\usepackage{bbm}
\usepackage{enumerate}
\usepackage{graphicx}
\usepackage{dsfont}
\usepackage{color}
\begin{document}

\newtheorem{theorem}[subsection]{Theorem}
\newtheorem{proposition}[subsection]{Proposition}
\newtheorem{lemma}[subsection]{Lemma}
\newtheorem{corollary}[subsection]{Corollary}
\newtheorem{conjecture}[subsection]{Conjecture}
\newtheorem{prop}[subsection]{Proposition}
\newtheorem{defin}[subsection]{Definition}

\numberwithin{equation}{section}
\newcommand{\mr}{\ensuremath{\mathbb R}}
\newcommand{\mc}{\ensuremath{\mathbb C}}
\newcommand{\dif}{\mathrm{d}}
\newcommand{\intz}{\mathbb{Z}}
\newcommand{\ratq}{\mathbb{Q}}
\newcommand{\natn}{\mathbb{N}}
\newcommand{\comc}{\mathbb{C}}
\newcommand{\rear}{\mathbb{R}}
\newcommand{\prip}{\mathbb{P}}
\newcommand{\uph}{\mathbb{H}}
\newcommand{\fief}{\mathbb{F}}
\newcommand{\majorarc}{\mathfrak{M}}
\newcommand{\minorarc}{\mathfrak{m}}
\newcommand{\sings}{\mathfrak{S}}
\newcommand{\fA}{\ensuremath{\mathfrak A}}
\newcommand{\mn}{\ensuremath{\mathbb N}}
\newcommand{\mq}{\ensuremath{\mathbb Q}}
\newcommand{\half}{\tfrac{1}{2}}
\newcommand{\f}{f\times \chi}
\newcommand{\summ}{\mathop{{\sum}^{\star}}}
\newcommand{\chiq}{\chi \bmod q}
\newcommand{\chidb}{\chi \bmod db}
\newcommand{\chid}{\chi \bmod d}
\newcommand{\sym}{\text{sym}^2}
\newcommand{\hhalf}{\tfrac{1}{2}}
\newcommand{\sumstar}{\sideset{}{^*}\sum}
\newcommand{\sumprime}{\sideset{}{'}\sum}
\newcommand{\sumprimeprime}{\sideset{}{''}\sum}
\newcommand{\sumflat}{\sideset{}{^\flat}\sum}
\newcommand{\shortmod}{\ensuremath{\negthickspace \negthickspace \negthickspace \pmod}}
\newcommand{\V}{V\left(\frac{nm}{q^2}\right)}
\newcommand{\sumi}{\mathop{{\sum}^{\dagger}}}
\newcommand{\mz}{\ensuremath{\mathbb Z}}
\newcommand{\leg}[2]{\left(\frac{#1}{#2}\right)}
\newcommand{\muK}{\mu_{\omega}}
\newcommand{\thalf}{\tfrac12}
\newcommand{\lp}{\left(}
\newcommand{\rp}{\right)}
\newcommand{\Lam}{\Lambda_{[i]}}
\newcommand{\lam}{\lambda}
\newcommand{\af}{\mathfrak{a}}
\newcommand{\sw}{S_{[i]}(X,Y;\Phi,\Psi)}
\newcommand{\lz}{\left(}
\newcommand{\pz}{\right)}
\newcommand{\bfrac}[2]{\lz\frac{#1}{#2}\pz}
\newcommand{\odd}{\mathrm{\ primary}}
\newcommand{\even}{\text{ even}}
\newcommand{\res}{\mathrm{Res}}
\newcommand{\sumn}{\sumstar_{(c,1+i)=1}  w\left( \frac {N(c)}X \right)}
\newcommand{\lab}{\left|}
\newcommand{\rab}{\right|}
\newcommand{\Go}{\Gamma_{o}}
\newcommand{\Ge}{\Gamma_{e}}
\newcommand{\M}{\widehat}
\def\su#1{\sum_{\substack{#1}}}

\theoremstyle{plain}
\newtheorem{conj}{Conjecture}
\newtheorem{remark}[subsection]{Remark}

\newcommand{\pfrac}[2]{\left(\frac{#1}{#2}\right)}
\newcommand{\pmfrac}[2]{\left(\mfrac{#1}{#2}\right)}
\newcommand{\ptfrac}[2]{\left(\tfrac{#1}{#2}\right)}
\newcommand{\pMatrix}[4]{\left(\begin{matrix}#1 & #2 \\ #3 & #4\end{matrix}\right)}
\newcommand{\ppMatrix}[4]{\left(\!\pMatrix{#1}{#2}{#3}{#4}\!\right)}
\renewcommand{\pmatrix}[4]{\left(\begin{smallmatrix}#1 & #2 \\ #3 & #4\end{smallmatrix}\right)}
\def\en{{\mathbf{\,e}}_n}

\newcommand{\ppmod}[1]{\hspace{-0.15cm}\pmod{#1}}
\newcommand{\ccom}[1]{{\color{red}{Chantal: #1}} }
\newcommand{\acom}[1]{{\color{blue}{Alia: #1}} }
\newcommand{\alexcom}[1]{{\color{green}{Alex: #1}} }
\newcommand{\hcom}[1]{{\color{brown}{Hua: #1}} }

\makeatletter
\def\widebreve{\mathpalette\wide@breve}
\def\wide@breve#1#2{\sbox\z@{$#1#2$}%
     \mathop{\vbox{\m@th\ialign{##\crcr
\kern0.08em\brevefill#1{0.8\wd\z@}\crcr\noalign{\nointerlineskip}%
                    $\hss#1#2\hss$\crcr}}}\limits}
\def\brevefill#1#2{$\m@th\sbox\tw@{$#1($}%
  \hss\resizebox{#2}{\wd\tw@}{\rotatebox[origin=c]{90}{\upshape(}}\hss$}
\makeatletter

\title[Bounds for integral moments of  quadratic twists of Fourier coefficients of modular forms]{Bounds for integral moments of quadratic twists of Fourier coefficients of modular forms}

\author[P. Gao]{Peng Gao}
\address{School of Mathematical Sciences, Beihang University, Beijing 100191, China}
\email{penggao@buaa.edu.cn}

\author[Y. Zhao]{Yuetong Zhao}
\address{School of Mathematical Sciences, Beihang University, Beijing 100191, China}
\email{yuetong.zhao.math@gmail.com}

\begin{abstract}
In this paper, we obtain under the assumption of the Generalized Riemann Hypothesis upper bounds for all high integral moments of sums of Fourier coefficients of a given modular form twisted by quadratic Dirichlet characters. We show the bounds are optimal by establishing unconditionally the corresponding lower bounds for all even integer moments.
\end{abstract}

\maketitle



\section{Introduction}\label{sec 1}

Let $\chi_d=\leg {d}{\cdot}$ denote the Kronecker symbol associated to any fundamental discriminant $d$.  Here we recall that $d$ is a fundamental discriminant if either $d$ is square-free with $d \equiv 1 \pmod 4$, or $d=4n$ where $n$ is square-free and $n \equiv 2,3 \pmod 4$. Estimating moments of quadratic Dirichlet character sums is an important subject in analytic number theory. For $m>0$, we define the $m$-th moment of quadratic character sums as
\begin{align*}
 S_m(X,Y):= \sumstar_{0<d \leq X }\Big | \sum_{n \leq Y}\chi_d(n)\Big |^{m},
\end{align*}
 where $m$, $X$, and $Y$ are positive real numbers, and we denote $\sumstar$ a sum over fundamental discriminants throughout the paper. A conjecture of M. Jutila \cite{Jutila} asserts that for any positive integer $m$, there exist constants $c_1(m)$ and $c_2(m)$, depending only on $m$, such that
 \begin{align*}
 S_m(X,Y)\leq c_1(m)XY^{m/2}(\log X)^{c_2(m)}.
 \end{align*}
 In \cite{G&Zhao2024},  P. Gao and L. Zhao confirmed a smoothed version of Jutila's conjecture  under the Generalized Riemann Hypothesis (GRH), using upper bounds for moments of quadratic Dirichlet $L$-functions. More precisely, they proved that for large $X$, $Y$ and any real $m\geq 1/2$,
\begin{align*}
 S_m(X,Y;W):=\sumstar_{\substack{0<d \leq X }}\Big | \sum_{n \geq 1}\chi_{d}(n)W\Big(\frac{n}{Y}\Big)\Big |^{m} \ll XY^{m/2}(\log X)^{\frac {m(m+1)}{2}},
\end{align*}
where $W$ is a smooth weight function.  Furthermore,  in
 \cite[Theorem 1.3]{G&Zhao2024-3}, they showed  that under GRH, for any real $m>\sqrt{5}+1$,
 \begin{align*}
 S_m(X,Y)\ll XY^{m/2}(\log X)^{\frac{m(m-1)}{2}+1},
 \end{align*}
  using estimates for shifted moments of quadratic Dirichlet $L$-functions. Subsequently,  M. Munsch \cite{Mun} improved the above bound for the smoothed version concerning integers $m$.  Under GRH, he proved that for any  integer $m\geq2$,
   \begin{align*}
 S_m(X,Y;W)\ll XY^{m/2}(\log X)^{\frac{m(m-1)}{2}},
 \end{align*}
   which reduces the  exponent from $\frac {m(m-1)}{2}+1$ to the optimal $\frac {m(m-1)}{2}$ for integer $m \geq 2$. This result is consistent with earlier work of M. V. Armon \cite[Theorem 2]{Armon}, who unconditionally established the optimal bound  for $S_1(X,Y)$.
\bigskip

Let $f$ be a fixed holomorphic Hecke eigenform of weight $\kappa\equiv 0 \pmod{4}$ for the full modular group $SL_2(\mz)$. The Fourier expansion of  $f$ at infinity is given by
\begin{align*}
f(z)=\sum_{n=1}^{\infty}\lambda_f(n)n^{\frac{\kappa-1}{2}}e(nz),
\end{align*}
where $e(z)=e^{2\pi iz}$.
 In this paper, we are interested in bounding moments of sums involving $\lambda_f (n)$ twisted by quadratic characters $\chi_{d}(n)$, where $d$ ranges over odd, positive, square-free integers.  More precisely, we define
\begin{align*}
 T_{m}(X,Y;f):=  \sumstar_{0<d \leq X }\Big | \sum_{n \leq Y}\chi_d(n)\lambda_f(n)\Big |^{m},
\end{align*}
 where $m$, $X$, and $Y$ are positive real numbers.
In \cite{G&Zhao24-12}, P. Gao and L. Zhao  proved that
\begin{align}\label{G&Z}
T_m(X,Y;f)\ll XY^{m/2}(\log X)^{\frac{(m-2)(m-1)}{2}+1}
\end{align}
for any real number $m\geq 4$.
In a subsequent work \cite{G&Zyt}, the authors \cite[Theorem 1.1]{G&Zyt} obtained an asymptotic formula in the smoothed case when $m=2$. Their result implies that
for two non-negative, smooth functions $\Phi, \Psi$ that are compactly supported functions on the positive real numbers, we have for large $X, Y$ such that $Y \ll X^{1-\varepsilon}$ for any $\varepsilon>0$,
\begin{align*}
    \sumstar_{d} \Bigg(\sum_{n} \lambda_f(n)\chi_{8d}(n)\Phi \Big(\frac nY \Big)\Bigg)^2 \Psi \Big(\frac d{X} \Big) \sim  XY.
\end{align*}
   Now, the above result together with \eqref{G&Z} seems to suggest that one expects to have for all real $m \geq 2$,
\begin{align*}
 T_{m}(X,Y;f)\ll XY^{m/2}(\log X)^{\frac {m(m-3)}{2}+1}.
\end{align*}

  It is the aim of this paper to show that the exponent $\frac {m(m-3)}{2}+1$ maybe further reduced to $\frac {m(m-3)}{2}$ under GRH for large integers $m$ and that this is bess possible. For simplicity, we consider the smoothed moments and we define
\begin{align*}
T_m(X,Y;f,\Phi):=\sumstar_{0<d\leq X}\Big|\sum_{n\geq 1}\chi_d(n)\lambda_f(n)\Phi\Big(\frac{n}{Y}\Big)\Big|^m,
 \end{align*}
where $\Phi$ is a non-negative, smooth and compacted supported function as mentioned above.

  Our first result concerns with the upper bound for $T_m(X,Y;f,\Phi)$.
\begin{theorem}\label{upper}
  With the notation as above and assuming the truth of GRH. For any integer $m\geq 4$ and sufficiently large $Y\leq X$, we have
  \begin{equation}
  \label{Tupper}
  T_m(X,Y;f,\Phi) \ll XY^{m/2} (\log X)^{\frac{m(m-3)}{2 }}.
  \end{equation}
\end{theorem}

   Our proof of Theorem \ref{upper} is motivated by the work of Munsch \cite{Mun} concerning the moments of quadratic Dirichlet character sums. Our proof also relies on a careful application of the bounds on shifted moments of twisted modular $L$-functions established in \cite[Corollary 1.4]{G&Zhao24-12}.

   We also follow the approach in \cite{Mun} to establish unconditionally a commensurate lower bound for $T_m(X,Y;f,\Phi)$.
\begin{theorem}\label{lower}
  With the notation as above. For any even integer $m\geq 4$ and sufficiently large $X$, we have for $X^{\varepsilon} \ll Y\ll X^{2/3-\alpha}$ for any $\alpha>0$,
\begin{equation*}
  T_m(X,Y;f,\Phi) \gg XY^{m/2} (\log X)^{\frac{m(m-3)}{2 }}.
\end{equation*}
 Moreover, under GRH, the above bound remains valid in the larger range  $X^{\varepsilon} \ll Y\ll X^{1-\alpha}$.
\end{theorem}

   In view of Theorem \ref{upper} and \ref{lower}, we see that the upper bound for $T_m(X,Y;f,\Phi)$ given in \eqref{Tupper} is indeed optimal for even integers $m \geq 4$.

\section{Preliminaries}\label{sec 2}

\subsection{Cusp form $L$-functions}
\label{sec:cusp form}

    We reserve the letter $p$ for a prime number throughout in this paper.  Recall that $f$ is a fixed holomorphic Hecke eigenform $f$ of weight $\kappa \equiv 0 \pmod 4$ for the full modular group $SL_2 (\mathbb{Z})$. The associated modular $L$-function $L(s, f)$ for $\Re(s)>1$ is then defined to be
\begin{align}
\label{Lphichi}
L(s, f ) &= \sum_{n=1}^{\infty} \frac{\lambda_f(n)}{n^s}
 = \prod_{p} \left(1 - \frac{\lambda_f (p)}{p^s}  + \frac{1}{p^{2s}}\right)^{-1}=\prod_{p} \left(1 - \frac{\alpha_p }{p^s} \right)^{-1}\left(1 - \frac{\beta_p }{p^s} \right)^{-1}.
\end{align}
 By Deligne's proof \cite{D} of the Weil conjecture, we know that
\begin{align*}
|\alpha_{p}|=|\beta_{p}|=1, \quad \alpha_{p}\beta_{p}=1.
\end{align*}
  It follows that $\lambda_f(n) \in \mr$ such that $\lambda_f (1) =1$ and
\begin{align}
\label{lambdabound}
\begin{split}
  |\lambda_f(n)| \leq d(n) \ll n^{\varepsilon}.
\end{split}
\end{align}
 where $d(n)$ is the number of positive divisors $n$. Here the last estimation follows from the bound (see \cite[Theorem 2.11]{MVa1}) that for any $\varepsilon>0$,
\begin{align}
\label{divisorbound}
\begin{split}
  d(n) \ll n^{\varepsilon}.
\end{split}
\end{align}

     The symmetric square $L$-function $L(s, \operatorname{sym}^2 f)$ of $f$ is defined for $\Re(s)>1$ by
 (see \cite[p. 137]{iwakow} and \cite[(25.73)]{iwakow})
\begin{align*}
\begin{split}
 L(s, \operatorname{sym}^2 f)=& \prod_p(1-\alpha^2_pp^{-s})^{-1}(1-p^{-s})^{-1}(1-\beta^2_pp^{-s})^{-1} = \zeta(2s) \sum_{n \geq 1}\frac {\lambda_f(n^2)}{n^s}=\prod_{p}(1-\frac {\lambda_f(p^2)}{p^s}+\frac {\lambda_f(p^2)}{p^{2s}}-\frac {1}{p^{3s}} )^{-1}.
\end{split}
\end{align*}
  It follows from a result of G. Shimura \cite{Shimura} that $L(s, \operatorname{sym}^2 f)$ has no pole at $s=1$. Moreover, the corresponding completed symmetric square $L$-function
\begin{align}
\label{SymsquareLfeqn}
 \Lambda(s, \operatorname{sym}^2 f)=& \pi^{-3s/2}\Gamma (\frac {s+1}{2})\Gamma (\frac {s+\kappa-1}{2}) \Gamma (\frac {s+\kappa}{2}) L(s, \operatorname{sym}^2 f)
\end{align}
  is entire and satisfies the functional equation
$\Lambda(s, \operatorname{sym}^2 f)=\Lambda(1-s, \operatorname{sym}^2 f)$.

\subsection{Mean Value Estimates}

  In this section, we include some mean value estimates that are needed in our proof of Theorem \ref{lower}. We denote $\square$ for a perfect square.
\begin{lemma}\label{character}
  For all positive integers $n$,  we have for any $\varepsilon>0$,
\begin{align*}
    \sumstar_{0<d \leq X} \chi_{d}(n)=\frac{6}{\pi^2}X\prod_{p|n}\left(\frac{p}{p+1}\right)\mathds{1}_{n=\square}+
\begin{cases}
  O(X^{1/2}n^{1/4}\log n), \quad \text{unconditionally}, \\
  O\left(X^{1/2+\varepsilon}n^{\varepsilon}\right), \quad \text{under GRH},
\end{cases}
\end{align*}
  where $\mathds{1}_{n=\square}$ denotes the indicator function of the square numbers.
  \end{lemma}
\begin{proof}
  The unconditional result follows from \cite[Lemma  4.1]{GSound} while the conditional one follows from \cite[Lemma 1]{DM24}.
  \end{proof}

    Our next result concerns with the analytical property of certain Dirichlet series.
\begin{lemma}\label{Dirichletseries}
  Define
  \begin{align}
  \label{f}
  \begin{split}
g(n_1, \ldots, n_{m}):= & \prod_{1\leq i\leq m}\lambda_f(n_i)\prod_{p|n_1\cdots n_m}(1+1/p)^{-1}\cdot \mathds{1}_{n_1\cdots n_m=\square}, \\
g^+(n_1, \ldots, n_{m}):= & \prod_{1\leq i\leq m}|\lambda_f(n_i)|\prod_{p|n_1\cdots n_m}(1+1/p)^{-1}\cdot \mathds{1}_{n_1\cdots n_m=\square}.
\end{split}
  \end{align}
We have
\begin{align}
 \label{Gdef}
  \begin{split}
G({\bf s}):=& \sum_{n_1, \ldots, n_{m}=1}^{\infty} \frac{g(n_1, \ldots, n_{m})}{n_1^{s_1} \cdots n_m^{s_{m}}} =  \left( \prod_{1 \leq j \leq m} L(2s_j, \operatorname{sym}^2f ) \prod_{1 \leq l_1 < l_2 \leq m} \zeta(s_{l_1} + s_{l_2})L (s_{l_1} + s_{l_2},\operatorname{sym}^2f)\right) E(s_1, \ldots, s_{m}), \\
 G^+({\bf s}):=& \sum_{n_1, \ldots, n_{m}=1}^{\infty} \frac{g^+(n_1, \ldots, n_{m})}{n_1^{s_1} \cdots n_m^{s_{m}}} =  \left( \prod_{1 \leq j \leq m} L(2s_j, \operatorname{sym}^2f ) \prod_{1 \leq l_1 < l_2 \leq m} \zeta(s_{l_1} + s_{l_2})L (s_{l_1} + s_{l_2},\operatorname{sym}^2f)\right) E^+(s_1, \ldots, s_{m}),
\end{split}
\end{align}
where \( E(s_1, \ldots, s_{m}), E^+(s_1, \ldots, s_{m})  \) are Dirichlet series absolutely convergent for \( \Re(s_j) > 1/4 \) for \( 1 \leq j \leq m \).
\end{lemma}
\begin{proof}
  As the proofs are similar, we focus on the case concerning $g$ in what follows. From \eqref{f}, we know that $g(n_1, \ldots, n_{m})$ is jointly multiplicative with respect to all $n_i, 1\leq i \leq m$. It follows that
\begin{align*}
&\sum_{n_1,\dots,n_m=1}^{\infty}\frac{g(n_1,\dots,n_m)}{n_1^{s_1}\cdots n_m^{s_m}} \\
=& \prod_{p}\left(\sum_{v_1,\dots,v_m=0}^{\infty}\frac{g(p^{v_1},\dots,p^{v_m})}{p^{v_1s_1+\cdots+v_ms_m}}\right) \\
=& \prod_p \left( 1 + \sum^{\infty}_{\begin{subarray}{c} v_1, \ldots, v_{m}=0 \\ v_1 + \cdots + v_{m} = 2 \end{subarray}} \left( 1 - \frac{1}{p+1} \right) \frac{\lambda_f(p^{v_1})\cdots \lambda_f(p^{v_m})}{p^{v_1 s_1 + \cdots + v_{m} s_{m}}}
+ \sum^{\infty}_{\substack{v_1, \dots, v_m = 0 \\ v_1 + \cdots + v_{m} \text{ even}\\v_1+\cdots+v_m\geq 4}} \left( 1 - \frac{1}{p+1} \right) \frac{\lambda_f(p^{v_1})\cdots \lambda_f(p^{v_m})}{p^{v_1 s_1 + \cdots + v_{m} s_{m}}}\right).
\end{align*}
Let $s_j=\sigma_j+it_j$. By putting \(\sigma = \min\{ \sigma_j \mid 1 \leq j \leq m \}\) and applying $|\lambda_f(p^v)|\leq d(p^v)=v+1$ by \eqref{lambdabound}, we see that
\[
\sum^{\infty}_{\substack{v_1, \dots, v_m = 0 \\ v_1 + \cdots + v_{m} \text{ even}\\v_1+\cdots+v_m\geq 4}} \left( 1 - \frac{1}{p+1} \right) \frac{\lambda_f(p^{v_1})\cdots \lambda_f(p^{v_m})}{p^{v_1 s_1 + \cdots + v_{m} s_{m}}} \ll_m \frac{1}{p^{4\sigma}}.
\]

Next we factor out the products of the Riemann zeta-functions to obtain
\begin{align*}
\sum_{n_1, n_2, \ldots, n_{2k}=1}^{\infty} \frac{g(n_1, \ldots, n_{2k})}{n_1^{s_1} \cdots n_{2k}^{s_{2k}}} = \left( \prod_{1 \leq j \leq m} L(2s_j, \operatorname{sym}^2f ) \prod_{1 \leq l_1 < l_2 \leq m} \zeta(s_{l_1} + s_{l_2})L (s_{l_1} + s_{l_2},\operatorname{sym}^2f)\right) E(s_1, \ldots, s_{m}).
\end{align*}
Then, we can calculate \( E(s_1, \ldots, s_{m}) \) as
\begin{align*}
&E(s_1, \ldots, s_k) \\
= &\prod_p \left( 1 + \sum^{\infty}_{\begin{subarray}{c} v_1, \ldots, v_{m}=0 \\ v_1 + \cdots + v_{m} = 2 \end{subarray}} \left( 1 - \frac{1}{p+1} \right) \frac{\lambda_f(p^{v_1})\cdots \lambda_f(p^{v_m})}{p^{v_1 s_1 + \cdots + v_{m} s_{m}}}
+ \sum^{\infty}_{\substack{v_1, \dots, v_m = 0 \\ v_1 + \cdots + v_{m} \text{ even}\\v_1+\cdots+v_m\geq 4}} \left( 1 - \frac{1}{p+1} \right) \frac{\lambda_f(p^{v_1})\cdots \lambda_f(p^{v_m})}{p^{v_1 s_1 + \cdots + v_{m} s_{m}}}\right)\\
&\quad \times \prod_{1 \leq j \leq m} \left( 1 - \frac{\lambda_f(p^2)}{p^{2s_j}}+\frac{\lambda_f(p^2)}{p^{4s_j}}-\frac{1}{p^{6s_j}} \right) \prod_{1 \leq l_1 < l_2 \leq m} \left( 1 - \frac{1}{p^{s_{l_1} + s_{l_2}}} \right)\left( 1 - \frac{\lambda_f(p^2)}{p^{s_{l_1}+s_{l_2}}}+\frac{\lambda_f(p^2)}{p^{2(s_{l_1}+s_{l_2})}}-\frac{1}{p^{3(s_{l_1}+s_{l_2})}} \right)\\
=& \prod_p \left( 1 + \sum_{1 \leq i \leq m} \left( 1 - \frac{1}{p+1} \right) \frac{\lambda_f(p^2)}{p^{2s_i}} + \sum_{1 \leq h_1 < h_2 \leq m} \left( 1 - \frac{1}{p+1} \right) \frac{\lambda_f^2(p)}{p^{s_{h_1} + s_{h_2}}} + O_m \left( \frac{1}{p^{4\sigma}} \right) \right)\\
&\quad \times \left( 1 - \sum_{1 \leq j \leq m} \frac{\lambda_f(p^2)}{p^{2s_j}} + O_m \left( \frac{1}{p^{4\sigma}} \right) \right) \left( 1 - \sum_{1 \leq l_1 < l_2 \leq m} \frac{1}{p^{s_{l_1} + s_{l_2}}}- \sum_{1 \leq l_1 < l_2 \leq m} \frac{\lambda_f(p^2)}{p^{s_{l_1} + s_{l_2}}} + O_m \left( \frac{1}{p^{4\sigma}} \right) \right)\\
=& \prod_p \left( 1 -  \sum_{1 \leq i \leq m} \frac{\lambda_f(p^2)}{p^{2s_i}(p+1)} - \sum_{1 \leq h_1 < h_2 \leq m} \frac{\lambda_f(p^2)}{p^{s_{h_1} + s_{h_2} }(p+1)} + O_m \left( \frac{1}{p^{4\sigma}} \right) \right).
\end{align*}
Applying \eqref{lambdabound} again, we see that the above Euler product is convergent absolutely for \(\Re(s_j) > 1/4\). This completes the proof of the lemma.
\end{proof}

   We now apply Lemma \ref{Dirichletseries} to establish a result on sums of $\prod_{i=1}^{m}\lambda_f(n_i)$ restricted on the set $n_1\cdots n_m=\square$.
\begin{lemma}
\label{sumoversquare}
  With the notation as above. Let $\beta_i >0$ and define for $m \geq 1$,
\begin{equation*}
 P_f(Y^{\beta_1}, \cdots, Y^{\beta_m}):=\sum_{\substack{n_1,\dots,n_m \geq 1\\ n_1\cdots n_m=\square}} \prod_{i=1}^{m}\lambda_f(n_i)\Phi\Big(\frac{n_i}{Y^{\beta_i}}\Big)\prod_{p|n_1\cdots n_m}(1+1/p)^{-1}.
\end{equation*}
  Then we have for $m \geq 4$,
\begin{align}
\label{Pfest}
 P_f(Y^{\beta_1}, \cdots, Y^{\beta_m}) \asymp Y^{\sum^m_{i=1}\beta_i/2}(\log Y)^{\frac {m(m-3)}{2}}.
\end{align}
\end{lemma}
\begin{proof}
  We define for $m \geq 1$,
\begin{align*}
 P^+_f(Y^{\beta_1}, \cdots, Y^{\beta_m}):=& \sum_{\substack{n_1,\dots,n_m \geq 1\\ n_1\cdots n_m=\square}}  \prod_{i=1}^{m}|\lambda_f(n_i)|\Phi\Big(\frac{n_i}{Y^{\beta_i}}\Big)\prod_{p|n_1\cdots n_m}(1+1/p)^{-1}.
\end{align*}

  We also denote $\widehat{\Phi}$ the Mellin transform of $\Phi$ and we note that repeated integration by parts implies that for any integer $E\geq 0$ and $\Re(s)\geq 1/2$,
\begin{align}
\label{Phibound}
\widehat{\Phi}(s)\ll (1+|s|)^{-E}.
\end{align}

   We then apply Mellin inversion to the variables $n_1,\dots, n_{m}$ to see that
\begin{align}
\label{Hsq0}
\begin{split}
P_f(Y^{\beta_1}, \cdots, Y^{\beta_m})=& \Big(\frac{1}{2\pi i}\Big)^{m}\int_{(2)}\cdots\int_{(2)}Y^{\sum^m_{i=1}u_i\beta_i}G(u_1,\dots,u_{m})
\widehat{\Phi}(u_1)\cdots\widehat{\Phi}(u_{m})
\dif u_1\cdots \dif u_{m}, \\
P^+_f(Y^{\beta_1}, \cdots, Y^{\beta_m})=& \Big(\frac{1}{2\pi i}\Big)^{m}\int_{(2)}\cdots\int_{(2)}Y^{\sum^m_{i=1}u_i\beta_i}G^+(u_1,\dots,u_{m})
\widehat{\Phi}(u_1)\cdots\widehat{\Phi}(u_{m})
\dif u_1\cdots \dif u_{m},
\end{split}
\end{align}
where $G, G^+$ are defined in \eqref{Gdef}.

  We now evaluate the main term of \eqref{Hsq0} by first moving the integrals there to $\Re(u_j)=1/2+\varepsilon_i$ for $\varepsilon_i>0$ and mutually distinct. In this process, we encounter no pole so that
\begin{align}
 \label{Hsq1}
\begin{split}& P_f(Y^{\beta_1}, \cdots, Y^{\beta_m})\\
=&\Big(\frac{1}{2\pi i}\Big)^{m}\int_{(1/2+\varepsilon_1)}\cdots\int_{(1/2+\varepsilon_m)}Y^{\sum^m_{i=1}u_i\beta_i}
\Bigg( \prod_{1 \leq j \leq m} L(2u_j, \operatorname{sym}^2f ) \prod_{1 \leq l_1 < l_2 \leq m} \zeta(u_{l_1} + u_{l_2})L (u_{l_1} + u_{l_2},\operatorname{sym}^2f)\Bigg) \\
&\times E(u_1, \ldots, u_{m})\widehat{\Phi}(u_1)\cdots\widehat{\Phi}(u_{m})\dif u_1\cdots \dif u_{m}, \\
& P^+_f(Y^{\beta_1}, \cdots, Y^{\beta_m})\\
=&\Big(\frac{1}{2\pi i}\Big)^{m}\int_{(1/2+\varepsilon_1)}\cdots\int_{(1/2+\varepsilon_m)}Y^{\sum^m_{i=1}u_i\beta_i}
\Bigg( \prod_{1 \leq j \leq m} L(2u_j, \operatorname{sym}^2f ) \prod_{1 \leq l_1 < l_2 \leq m} \zeta(u_{l_1} + u_{l_2})L (u_{l_1} + u_{l_2},\operatorname{sym}^2f)\Bigg) \\
&\times E^+(u_1, \ldots, u_{m})\widehat{\Phi}(u_1)\cdots\widehat{\Phi}(u_{m})\dif u_1\cdots \dif u_{m}.
\end{split}
\end{align}
  Here we note that $E(u_1, \ldots, u_{m}), E^+(u_1, \ldots, u_{m})$ converge absolutely in the region of $\Re(u_j)$ larger than $1/4$, and are uniformly bounded there.

  We note the convexity bound for $\zeta(s), L(s,\operatorname{sym}^2f)$ as given in \cite[(5.20)]{iwakow} together with \eqref{SymsquareLfeqn} (which implies that the analytic conductor of $L(s,\operatorname{sym}^2f)$ is $\ll (1+|s|)^3$) to see that for any $\varepsilon>0$,
\begin{align*}
   \zeta(s) \ll
   (1+|s|)^{(1-\Re(s))/2+\varepsilon}, \  L(s,\operatorname{sym}^2f) \ll
   (1+|s|)^{3(1-\Re(s))/2+\varepsilon},\qquad & 0< \Re(s) <1.
\end{align*}

    The above estimations together with \eqref{Phibound} now allows us to shift the lines of integrations \eqref{Hsq1} in $u_m$ to $\Re(u_m)=1/4+\varepsilon$,  encountering $m-1$ simple poles for each case at $u_m=1-u_{\ell}$ for $\ell=1,\dots, m-1$ as the $\varepsilon_i$ are mutually distinct. We want to compute the main terms arising from this process and it is easy to see that the integrals on the new line contribute to only error terms. Thus, in both cases, the main term only comes from the sum of the residues of the $m-1$ simple poles, which involves with a multiple integral of $m-1$ variables. We then shift each such multiple integral in $u_{m-1}$ to $\Re(u_{m-1})=1/4+\varepsilon$, encountering $2(m-2)$ simple poles at $u_{m-1}=1-u_{s}, u_s$ for $s=1,\dots, m-2$. Again the main terms now only come from the sum of the corresponding residues, which  involve with a multiple integral of $m-1$ variables. We repeat the above process to see that, upon denoting $I=\{1, \cdots, m\}$,
\begin{align}
 \label{Hsq2}
\begin{split}
 P_f(Y^{\beta_1}, \cdots, Y^{\beta_m}) \sim & \frac{1}{2\pi i}\int_{(1/2+\varepsilon_1)}\sum_{S}Y^{\sum_{i \in S}u_1\beta_i+\sum_{i \in I\backslash S}(1-u_1)\beta_i} \times R \times E(v_1, \ldots, v_{m})\dif u_1, \\
P^+_f(Y^{\beta_1}, \cdots, Y^{\beta_m}) \sim & \frac{1}{2\pi i}\int_{(1/2+\varepsilon_1)}\sum_{S}Y^{\sum_{i \in S}u_1\beta_i+\sum_{i \in I\backslash S}(1-u_1)\beta_i} \times R \times E^+(v_1, \ldots, v_{m})\dif u_1,
\end{split}
\end{align}
  where $S$ ranges over all subsets of $I$, and where $v_i=u_1$ if $i \in S$ and $v_i=1-u_1$ otherwise. Also, $R$ is a product involving with $\zeta(2u_1), \zeta(2-2u_1), L(2u_1, \operatorname{sym}^2f ), L(2-2u_1, \operatorname{sym}^2f )$.

  We then shift the lines of integrations \eqref{Hsq2} in $u_1$ to $\Re(u_1)=1/4+\varepsilon$ to see that the main terms come from the sum of the residues at $u_1=1/2$, so that
\begin{align}
\label{Hsq3}
\begin{split}
 P_f(Y^{\beta_1}, \cdots, Y^{\beta_m}) \sim & CY^{\sum_{i \in I}\beta_i/2} (\log Y)^k  E(1/2, \ldots, 1/2), \\
P^+_f(Y^{\beta_1}, \cdots, Y^{\beta_m}) \sim & CY^{\sum_{i \in I}\beta_i/2} (\log Y)^k  E^+(1/2, \ldots, 1/2),
\end{split}
\end{align}
  where $C$ is some constant depending on $\widehat{\Phi}$ only.

  It thus remains to determine $k$. We shall do so for $P^+_f(Y^{\beta_1}, \cdots, Y^{\beta_m})$ by noting that $|\lambda_f(n)| \geq 0$. We are therefore able to apply the de la Bret\`{e}che Tauberian theorem \cite{Breteche}, following the arguments given in the proof of \cite[Theorem 1.3]{Toma}. Here, we modify the definition of $H({\bf s})$ given in \cite[(4.1)]{Toma} in our case to be
\begin{align*}
H({\bf s}):=G^+({\bf s} + {\bf 1/2})\left( \prod_{1 \leq l_1 <l_2 \leq 2k}(s_{l_1}+s_{l_2}) \right).
\end{align*}
  Using the notations given on \cite[Theorem 4.1--4.2]{Toma},  we see that in our case $\mathcal{L} = \{s_{l_1}+s_{l_2} \mid 1 \leq l_1<l_2 \leq m \}$, so that the cardinality $q$ of $\mathcal{L}$ equals $m(m-1)/2$. Moreover, we have $w=0$ in our case as well and the rank of the linear forms in $\mathcal{L}$ is $m$. It thus follows from \cite[(iv) of Th\'{e}or\`{e}me 2]{Breteche} that we have
\begin{align*}
 k=\frac {m(m-1)}{2}-m=\frac {m(m-3)}{2}.
\end{align*}
  We substitute the above value of $k$ into \eqref{Hsq3} to see that the desired relation given in \eqref{Pfest} is valid. This completes the proof of the lemma.
\end{proof}

\subsection{Shifted moments of quadratic twists of modular $L$-functions }
  We record here a result from
\cite{G&Zhao24-12} concerning bounds on shifted moments of quadratic twists of modular $L$-functions.
\begin{proposition}\label{prop}
With the notation as above and assuming the truth of GRH. Let $k\geq 1$ be a fixed integer and $A>0$ a fixed constant. Suppose $X$ is a large real number and $t=(t_1,\dots,t_k)$ is a real $k$-tuple with $|t_j|\leq X^A$. Then
\begin{align*}
  \sumstar_{0<d \leq X}\prod_{1\leq j\leq k}\big| L\big(1/2+it_j,f \otimes \chi_d \big) \big|^{a_j}
 \ll  X(\log X)^{\frac{1}{4}\sum_{j=1}^k a_j^2}  \mathcal{G}_1\cdot\mathcal{G}_2,
\end{align*}
where
\begin{align*}
\mathcal{G}_1=\prod_{1\leq i<j\leq k} g_1(|t_i-t_j|)^{a_ia_j/2}g_1(|t_i+t_j|)^{a_ia_j/2}\prod_{1\leq i\leq k} g_1(|2t_i|)^{a^2_i/4-a_i/2},
\end{align*}
\begin{align*}
\mathcal{G}_2=\prod_{1\leq i<j\leq k} g_2(|t_i-t_j|)^{a_ia_j/2}g_2(|t_i+t_j|)^{a_ia_j/2}\prod_{1\leq i\leq k} g_2(|2t_i|)^{a^2_i/4+a_i/2},
\end{align*}
and the functions $g_1, g_2 : \mathbb{R}_{\geq 0} \rightarrow \mathbb{R}$ are defined by
\begin{align}
\label{g1Def}
\begin{split}
g_1(x) =\begin{cases}
\log X,  & \text{if } x\leq 1/\log X \text{ or } x \geq e^X, \\
1/x, & \text{if }   1/\log X \leq x\leq 10, \\
\log \log x, & \text{if }  10 \leq x \leq e^{X},
\end{cases},
\quad
g_2(x) =
\begin{cases}
1,  & \text{if } x \leq e^e, \\
\log \log x, & \text{if }   e^e \leq x \leq e^{X}, \\
 \log X, & \text{if }  x \geq e^{X}.
\end{cases}
\end{split}
\end{align}

 Here the implied constant depends on $k$, $A$ and the $a_j$'s, but not on $X$ or the $t_j$'s.
\end{proposition}

\section{Proof of Theorem \ref{upper}}

   By a straightforward modification of the proof of \cite[Theorem 1.5]{G&Zhao24-12} together with the estimation given in \eqref{Phibound}, we see that under GRH,
\begin{align*}
 T_m (X,Y;f,\Phi) \ll &
   Y^{m/2}\sumstar_{0<d\leq X}\Bigg| \int\limits_{ |t|\leq X^{2\varepsilon}}\big|L( 1/2+it, f \otimes\chi_d)\big|\frac{1}{(1+|t|)^{10}} \dif t\Bigg|^{m}+O(XY^{m/2}).
\end{align*} 	

   It thus remains to show that under GRH, we have for any integer $m \geq 4$,
\begin{equation*}
 \sumstar_{0<d \leq X }
   \Bigg| \int\limits_{ |t|\leq X^{2\varepsilon}}\big|L( 1/2+it, f \otimes\chi_d)\big|\frac{1}{(1+|t|)^{10}} \dif t\Bigg|^{m} \ll X(\log X)^{\frac{m(m-3)}{2}}.
\end{equation*}

   Again similar to the proof of \cite[Lemma 5.1]{G&Zhao24-12}, we see that in order to establish the above estimation, it suffices to prove a refinement of  \cite[Proposition 5.4]{G&Zhao24-12} concerning the integral
\begin{equation*}
    T_m(B,X):=\sumstar_{0<d \leq X }\bigg( \int_{0}^{B} |L(1/2+it,f\otimes\chi_d) | \dif t \bigg)^{m},
\end{equation*}
where $10 \leq B=X^{O(1)}$.

   More precisely, we need to show that under GRH that for any integer $m \geq 4$,
\begin{align}
\label{Tbound}
    T_m(B,X) \ll B^2(\log \log B)^{O_m(1)}X(\log X)^{\frac{m(m-3)}{2}},
\end{align}

   In what follows, we shall establish \eqref{Tbound}. Our approach is motivated by the proof of \cite[Proposition 20]{Mun}.
We note first that we have
\begin{align*}
 &T_m(B,X) \\
 =& \sumstar_{0<d \leq X }\bigg( \int_{0}^{1/(2\log X)}|L(1/2+it,f\otimes\chi_d) |\dif t+ \int_{1/(2\log X)}^{5}|L(1/2+it,f\otimes\chi_d) |\dif t+ \int_{5}^{B}|L(1/2+it,f\otimes\chi_d) |\dif t \bigg)^{m} \\
 \ll &  \sum^3_{k=1}T_{m,k}(B,X),
\end{align*}
  where
\begin{align*}
 T_{m,1}(B,X)=& \sumstar_{0<d \leq X }\bigg( \int_{0}^{1/(2\log X)}|L(1/2+it,f\otimes\chi_d) |\dif t\bigg)^{m}, \\
 T_{m,2}(B,X)=& \sumstar_{0<d \leq X }\bigg( \int_{1/(2\log X)}^{5}|L(1/2+it,f\otimes\chi_d) |\dif t\bigg)^{m}, \\
 T_{m,3}(B,X)=& \sumstar_{0<d \leq X }\bigg(\int_{5}^{B}|L(1/2+it,f\otimes\chi_d) |\dif t \bigg)^{m}.
\end{align*}
 We have by symmetry that for $1 \leq k \leq 3$,
\begin{equation*}
 T_{m,k}(B,X)
      \ll  \sumstar_{0<d \leq X}\int\limits_{\mathcal{A}_{B,k}}\prod_{a=1}^m|L(1/2+ it_a, f\otimes\chi_d)| \dif \mathbf{t},
\end{equation*}
where $\mathcal{A}_{B,k}=\{ (t_1,\dots,t_m) \in I_k^m: 0 \leq t_1 \leq t_2 \dots \leq t_m\}$. Here we define $I_1=[0,1/(2\log X)], I_2=[1/(2\log X),5], I_3=[5, B]$.

Summing over $d$ and applying Proposition \ref{prop}, we see that for $1 \leq k \leq 3$,
\begin{align}
\label{TBXbound}
\begin{split}
 T_{m,k}(B,X)\ll& X(\log X)^{m/4} \int\limits_{\mathcal{A}_{B,k}} \prod_{1\leq i<j\leq m} g_1(|t_i-t_j|)^{1/2}g_1(|t_i+t_j|)^{1/2}g_2(|t_i-t_j|)^{1/2}g_2(|t_i+t_j|)^{1/2}\\
 &\qquad\qquad\qquad\qquad\times \prod_{1\leq i\leq m} g_1(|2t_i|)^{-1/4}g_2(|2t_i|)^{3/4}  \dif \mathbf{t}  \\
\ll & X(\log X)^{m/4}(\log \log B)^{O_m(1)} \int\limits_{\mathcal{A}_{B,k}} \prod_{1\leq i<j\leq m} g_1(|t_i-t_j|)^{1/2}g_1(|t_i+t_j|)^{1/2} \prod_{1\leq i\leq m} g_1(|2t_i|)^{-1/4}  \dif \mathbf{t} \\
:= & X(\log \log B)^{O_m(1)}I_{m,B,k},
\end{split}
\end{align}
  where the second estimation above follows by bounding $g_2$ trivially by $\log \log B$ using \eqref{g1Def}, and where
\begin{align*}
I_{m,B,k} = &(\log X)^{m/4}  \int\limits_{\mathcal{A}_{B,k}} \prod_{1\leq i<j\leq m} g_1(|t_i-t_j|)^{1/2}g_1(|t_i+t_j|)^{1/2} \prod_{1\leq i\leq m} g_1(|2t_i|)^{-1/4}\dif \mathbf{t}.
\end{align*}

  We now proceed to show that for all integers $m\geq 4$ and $1 \leq k \leq 3$,
\begin{equation}
\label{ind_hyp}
I_{m,B,k} \ll B^2 (\log X)^{\frac{m(m-3)}{2}}(\log \log B)^{O_m(1)}.
\end{equation}
  Without loss of generality, we may assume that $0 \leq t_1 \leq t_2 \leq \dots \leq t_m$ in what follows.

  If $t_i \in I_{B,1}$,  we have by \eqref{g1Def} that $g_1(|t_i \pm t_j|)=g_1(|2t_i|)=\log X$ for $1\leq i<j\leq m$. It allows us to see that
\begin{align}
\label{T1bound}
I_{m,B,1} \ll & (\log X)^{m(m-1)/2} \int\limits_{\mathcal{A}_{B,1}}\dif \mathbf{t} \ll  (\log X)^{m(m-1)/2-m} =(\log X)^{\frac{m(m-3)}{2}}.
\end{align}

  If $t_i \in I_{B,2}$,  we note from \eqref{g1Def} and the fact that the $t_i$'s are increasing that
$$ g_1(|t_i +t_j|)= (t_i+t_j)^{-1}\ll t_j^{-1}, \quad g_1(|2t_i|) = (2t_i)^{-1}.$$
  We deduce from the above that
\begin{align}
\label{Int4B}
\begin{split}
  I_{m,B,2} \ll & (\log X)^{m/4} \int_{1/(2\log X)}^{5}\int_{t_1}^{5}\int_{t_2}^{5}\cdots \int_{t_{m-1}}^{5}  (t_1\cdots t_m)^{1/4}\prod_{1\leq i<j\leq m}\frac{1}{(t_i+t_j)^{1/2}}\prod_{1\leq i<j\leq m} g_1(|t_i-t_j|)^{1/2}\dif t_m \cdots \dif t_1\\
  \ll & (\log X)^{m/4} \int_{1/(2\log X)}^{5}\int_{t_1}^{5}\int_{t_2}^{5}\cdots \int_{t_{m-1}}^{5}(t_1\cdots t_m)^{1/4}\prod_{1\leq i<j\leq m}t_j^{-1/2}\prod_{1\leq i<j\leq m} g_1(|t_i-t_j|)^{1/2}\dif t_m \cdots \dif t_1 \\
  = & (\log X)^{m/4} \int_{1/(2\log X)}^{5}\int_{t_1}^{5}\int_{t_2}^{5}\cdots \int_{t_{m-1}}^{5}(t_1\cdots t_m)^{1/4}\prod_{2\leq j\leq m}t_j^{-(j-1)/2}\prod_{1\leq i<j\leq m} g_1(|t_i-t_j|)^{1/2}\dif t_m \cdots \dif t_1 \\
  = & (\log X)^{m/4} \int_{1/(2\log X)}^{5}\int_{t_1}^{5}\int_{t_2}^{5}\cdots \int_{t_{m-1}}^{5}t_1^{1/4}\prod_{2\leq j\leq m}t_j^{-(j-1)/2+1/4}\prod_{1\leq i<j\leq m} g_1(|t_i-t_j|)^{1/2}\dif t_m \cdots \dif t_1 \\
  \ll & (\log X)^{m/4} \int_{1/(2\log X)}^{5}\int_{t_1}^{5}\int_{t_2}^{5}\cdots \int_{t_{m-1}}^{5}t_1^{1/4}\prod_{2\leq j\leq m}t_2^{-(j-1)/2+1/4}\prod_{1\leq i<j\leq m} g_1(|t_i-t_j|)^{1/2}\dif t_m \cdots \dif t_1 \\
   = & (\log X)^{m/4} \int_{1/(2\log X)}^{5}\int_{t_1}^{5}\int_{t_2}^{5}\cdots \int_{t_{m-1}}^{5}t_1^{1/4}t_2^{-(m-1)^2/4}\prod_{1\leq i<j\leq m} g_1(|t_i-t_j|)^{1/2}\dif t_m \cdots \dif t_1.
\end{split}
\end{align}

   We now evaluate the integral
\begin{align}
\label{Int4}
\begin{split}
  & \int_{t_{m-1}}^{5}  \prod_{1\leq i< m} g_1(|t_m-t_i|)^{1/2}\dif t_m \\
  =& \int_{t_{m-1}}^{t_{m-1}+1/\log X} \prod_{1\leq i< m} g_1(|t_m-t_i|)^{1/2}\dif t_m+\int_{t_{m-1}+1/\log X}^{5}  \prod_{1\leq i< m} g_1(|t_m-t_i|)^{1/2}\dif t_m.
\end{split}
\end{align}
  For the first integral on the right-hand side above, we apply the bounds $g_1(|t_m-t_i|) \leq \log X$ for $1\leq i<m$ from \eqref{g1Def} to see that
\begin{align}
\label{Int41}
\begin{split}
  \int_{t_{m-1}}^{t_{m-1}+1/\log X} \prod_{1\leq i< m} g_1(|t_m-t_i|)^{1/2}\dif t_m \ll (\log X)^{(m-1)/2} \int_{t_{m-1}}^{t_{m-1}+1/\log X}  1\dif t_m
  \ll (\log X)^{(m-1)/2-1}.
\end{split}
\end{align}

   When $t_m \geq t_{m-1}+1/\log X$, we apply the relation $g_1(|t_m-t_{m-1}|) \leq (t_m-t_{m-1})^{-1}$. As $t_{m-1} \geq t_i$ for all $1 \leq i < m-1$, it follows that we have $g_1(|t_m-t_i|) = (t_m-t_i)^{-1} \leq (t_m-t_{m-1})^{-1}$ for all $1 \leq i <m$. We then deduce that
\begin{align}
\label{Int42}
\begin{split}
  & \int_{t_{m-1}+1/\log X}^{5}  \prod_{1\leq i< m} g_1(|t_m-t_i|)^{1/2}\dif t_m
  \ll  \int_{t_{m-1}+1/\log X}^{5}  (t_m-t_{m-1})^{-(m-1)/2}\dif t_m \ll  (\log X)^{(m-1)/2-1}.
\end{split}
\end{align}

   Substituting the estimations obtained in \eqref{Int41} and \eqref{Int42} into \eqref{Int4}, we see that
\begin{align*}
\begin{split}
  & \int_{t_{m-1}}^{5}  \prod_{1\leq i< m} g_1(|t_m-t_i|)^{1/2}\dif t_m \ll (\log X)^{(m-1)/2-1}.
\end{split}
\end{align*}

   It follows from this and \eqref{Int4B} that
\begin{align}
\label{Int4B3}
\begin{split}
  I_{m,B,2} \ll & (\log X)^{m/4}(\log X)^{(m-1)/2-1}\int_{1/(2\log X)}^{5}\int_{t_1}^{5}\int_{t_2}^{5}\cdots \int_{t_{m-2}}^{5}t_1^{1/4}t_2^{-(m-1)^2/4}\prod_{1\leq i<j\leq m-1} g_1(|t_i-t_j|)^{1/2}\dif t_{m-1} \cdots \dif t_1.
\end{split}
\end{align}

   We now argue similar to those done in  \eqref{Int4}--\eqref{Int42} to see that
\begin{align*}
\begin{split}
   \int_{t_3}^{5}\cdots \int_{t_{m-2}}^{5}\prod_{1\leq i<j\leq m-1} g_1(|t_i-t_j|)^{1/2}\dif t_{m-1} \cdots \dif t_4  \ll &  (\log X)^{\sum^{m-3}_{i=1}(\frac {m-i}2-1)}, \\
   \int_{t_2}^{5}  \prod_{1\leq i< 3} g_1(|t_3-t_i|)^{1/2}\dif t_3 \ll & \log \log X \ll  \log \log B.
\end{split}
\end{align*}

   We deduce from the above and \eqref{Int4B3} that
\begin{align}
\label{Int4B2}
\begin{split}
  I_{m,B,2} \ll & (\log X)^{m/4}(\log X)^{\sum^{m-2}_{i=1}(\frac {m-i}2-1)}\log \log B\int_{1/(2\log X)}^{5}\int_{t_1}^{5}  t_1^{1/4}t_2^{-(m-1)^2/4} g_1(|t_2-t_1|)^{1/2}\dif t_2 \dif t_1 \\
  \ll &  (\log X)^{m/4+m(m-1)/4-(m-1)+1/2}\log \log B\int_{1/(2\log X)}^{5}\int_{t_1}^{5}  t_1^{1/4}t_2^{-(m-1)^2/4} g_1(|t_2-t_1|)^{1/2}\dif t_2 \dif t_1
\end{split}
\end{align}

   We argue again similar to those done in  \eqref{Int4}--\eqref{Int42} to see that
\begin{align}
\label{Int2est}
\begin{split}
 \int_{t_1}^{5} t_2^{-(m-1)^2/4} g_1(|t_2-t_1|)^{1/2}\dif t_2 \ll  &
  \int_{t_1}^{t_{1}+1/\log X}  t_1^{-(m-1)^2/4} (\log X)^{1/2}dt_2+\int_{t_{1}+1/\log X}^5  t_2^{-(m-1)^2/4} (t_2-t_1)^{-1/2}dt_2 \\
  \ll  &  t_1^{-(m-1)^2/4} (\log X)^{-1/2}+\int_{1/\log X}^{5}  (u+t_1)^{-(m-1)^2/4}u^{-1/2}du.
\end{split}
\end{align}

   We estimate the last integral above using the estimations $(u+t_1)^{-(m-1)^2/4} \ll t_1^{-(m-1)^2/4}$ when $u \leq t_1$ and $(u+t_1)^{-(m-1)^2/4} \ll u^{-(m-1)^2/4}$ to see that for $m \geq 3$,
\begin{align*}
\begin{split}
 \int_{1/\log X}^{5}  (u+t_1)^{-(m-1)^2/4}u^{-1/2}du \ll \int_{1/\log X}^{t_1}t_1^{-(m-1)^2/4}u^{-1/2}du+\int_{t_1}^{5}u^{-(m-1)^2/4-1/2}du \ll t_1^{-(m-1)^2/4+1/2} .
\end{split}
\end{align*}

   We apply the above estimation in \eqref{Int2est} to see that
\begin{align*}
\begin{split}
 \int_{t_1}^{5} t_2^{-(m-1)^2/4} g_1(|t_2-t_1|)^{1/2}\dif t_2 \ll &  t_1^{-(m-1)^2/4} (\log X)^{-1/2}+t_1^{-(m-1)^2/4+1/2}.
\end{split}
\end{align*}

   We deduce from the above and \eqref{Int4B2} that for $m \geq 4$,
\begin{align}
\label{Int4B1}
\begin{split}
  I_{m,B,2} \ll & (\log X)^{m/4+m(m-1)/4-(m-1)+1/2}\log \log B\int_{1/(2\log X)}^{5}t_1^{-(m-1)^2/4+1/4} (\log X)^{-1/2}+t_1^{-(m-1)^2/4+1/4+1/2})\dif t_1 \\
  \ll & (\log X)^{m/4+m(m-1)/4-(m-1)+1/2}(\log X)^{(m-1)^2/4-1/4-1-1/2}\log \log B \\
  =& (\log X)^{m(m-1)/2-(m-1)-1}\log \log B \\
  =& (\log X)^{\frac {m(m-3)}{2}}\log \log B.
\end{split}
\end{align}

   If $t_i \in I_{B,3}$,  we use the bounds $g_1(|t_i+t_j|)\ll (\log \log B)$ for $1\leq i<j\leq m$.  Note also that we have $g_1(|2t_i|) \gg \log \log 10$ in this case, so that
\begin{align}
\label{IB3bound}
I_{m,B,3} \ll & (\log X)^{m/4} (\log \log B)^{O_m(1)}\int\limits_{\mathcal{A}_{B,3}}\prod_{1\leq i<j\leq m} g_1(|t_i-t_j|)^{1/2}\dif \mathbf{t}.
\end{align}

    We argue as the case $t_i \in I_{B,2}$ to see that
\begin{align}
\begin{split}
\label{I3}
 \int\limits_{\mathcal{A}_{B,3}}\prod_{1\leq i<j\leq m} g_1(|t_i-t_j|)^{1/2}\dif \mathbf{t} \ll & (\log X)^{\sum^{m-2}_{i=1}(\frac {m-i}2-1)}\int_{5}^{B}\int_{t_1}^{B}g_1(|t_2-t_1|)^{1/2} \dif t_2\dif t_1 \\
 =& (\log X)^{m(m-1)/4-(m-1)+1/2}\int_{5}^{B}\int_{t_1}^{B}g_1(|t_2-t_1|)^{1/2} \dif t_2\dif t_1.
\end{split}
\end{align}

   Similar to our arguments given for the case $t_i \in I_{B,2}$, we have that
\begin{align*}
\begin{split}
  & \int_{t_1}^{B} g_1(|t_2-t_1|)^{1/2}\dif t_2 \ll  B.
\end{split}
\end{align*}
  It follows from this and \eqref{I3} that
\begin{align*}
 \int\limits_{\mathcal{A}_{B,3}}\prod_{1\leq i<j\leq m} g_1(|t_i-t_j|)^{1/2}\dif \mathbf{t} \ll (\log X)^{m(m-1)/4-(m-1)+1/2}B^{2},
\end{align*}
   We apply the above estimation in \eqref{IB3bound} to see that when $m \geq 4$,
\begin{align}
\label{I3bound}
\begin{split}
I_{m,B,3} \ll & (\log X)^{m/4+m(m-1)/4-(m-1)+1/2} B^{2}(\log \log B)^{O_m(1)} \\
           =& (\log X)^{m^2/4-m+3/2} B^{2}(\log \log B)^{O_m(1)} \\
           \ll & (\log X)^{m(m-3)/2} B^{2}(\log \log B)^{O_m(1)}.
\end{split}
\end{align}

   The desired estimation in \eqref{Tbound} now follows from \eqref{TBXbound}--\eqref{T1bound}, \eqref{Int4B1} and \eqref{I3bound}. This completes
   the proof of Theorem \ref{upper}.

\section{Proof of Theorem \ref{lower}}

   We denote for any $Y\geq 1$ and $0<d\leq X$,
$$H_{d}(Y;f,\Phi):= \sum_{n \geq 1} \chi_{d}(n)\lambda_f(n)\Phi\Big(\frac{n}{Y}\Big).$$

   We further define
\begin{equation*}
\mathcal{H}_1= \sumstar_{0<d\leqslant X} H_{d}(Y;f,\Phi) (H_{d}(Y^{\varepsilon};f,\Phi))^{m-1}, \quad \quad
\mathcal{H}_2=\sumstar_{0<d\leqslant X}\vert H_{d}(Y^{\varepsilon};f,\Phi) \vert^m.
\end{equation*}
By H\"older inequality, we get
\begin{align}
\label{Holder}
 \mathcal{H}_1^{m} \leq
 \mathcal{H}_2^{m-1} T_m(X,Y;f,\Phi).
 \end{align}

 We evaluate $ \mathcal{H}_2$ by noting that $\lambda_f(n) \in \mr$ and that $m$ is even, so that by Lemma \ref{character},
  \begin{align}
\label{eqS2}
\begin{split}
 \mathcal{H}_2 = &\sumstar_{0<d\leqslant X}( H_{d}(Y^{\varepsilon};f,\Phi) )^m \\
=& \sumstar_{0<d\leqslant X} \,\,\sum_{n_1,\dots,n_m \geq 1} \chi_{d}(n_1\cdots n_m)\lambda_f(n_1)\cdots \lambda_f(n_m)\Phi\Big(\frac{n_1}{Y^{\varepsilon}}\Big)\cdots \Phi\Big(\frac{n_m}{Y^{\varepsilon}}\Big) \\
  = & \frac 6{\pi^2}X\sum_{\substack{n_1,\dots,n_m \geq 1\\ n_1\cdots n_m=\square}} \lambda_f(n_1)\cdots \lambda_f(n_m)\Phi\Big(\frac{n_1}{Y^{\varepsilon}}\Big)\cdots \Phi\Big(\frac{n_m}{Y^{\varepsilon}}\Big)\prod_{p|n_1\cdots n_m}(1+1/p)^{-1}\\
& +O(X^{1/2}\sum_{n_1,\dots,n_m \ll Y^{\varepsilon}}|\lambda_f(n_1)\cdots \lambda_f(n_m)|(n_1\dots n_m)^{1/4+\varepsilon}) \\
= & \frac 6{\pi^2}X\sum_{\substack{n_1,\dots,n_m \geq 1\\ n_1\cdots n_m=\square}} \lambda_f(n_1)\cdots \lambda_f(n_m)\Phi\Big(\frac{n_1}{Y^{\varepsilon}}\Big)\cdots \Phi\Big(\frac{n_m}{Y^{\varepsilon}}\Big)\prod_{p|n_1\cdots n_m}(1+1/p)^{-1}+O(X^{1/2}\sum_{n_1,\dots,n_m \ll Y^{\varepsilon}}(n_1\dots n_m)^{1/4+\varepsilon}),
\end{split}
\end{align}
where the last estimation above follows from \eqref{divisorbound}.

  Note that
\begin{align}\label{S21}
  X^{1/2} \sum_{n_1,\dots,n_m \ll Y^{\varepsilon}}(n_1\dots n_m)^{1/4+\varepsilon} \ll X^{1/2} Y^{(\frac{5}{4}+\varepsilon)m \varepsilon}=X^{1/2+\varepsilon}.
\end{align}
  Moreover, we have by Lemma \ref{sumoversquare} that for integers $m \geq 4$,
\begin{align}\label{S22}
 \sum_{\substack{n_1,\dots,n_m \geq 1\\ n_1\cdots n_m=\square}} \lambda_f(n_1)\cdots \lambda_f(n_m)\Phi\Big(\frac{n_1}{Y^{\varepsilon}}\Big)\cdots \Phi\Big(\frac{n_m}{Y^{\varepsilon}}\Big)\prod_{p|n_1\cdots n_m}(1+1/p)^{-1}
& \ll Y^{m\varepsilon/2}(\log Y)^{\frac{m(m-3)}{2}} .
\end{align}
  We deduce from $\eqref{eqS2}$--$\eqref{S22}$ that we have
\begin{align}
\label{H2bound}
\mathcal{H}_2 \ll_{\varepsilon, m} X Y^{m\varepsilon/2}(\log X)^{\frac{m(m-3)}{2}}.
\end{align}

  Next we note that by Lemma \ref{character} and arguing similar to our discussions concerning $\mathcal{H}_2$, we have
\begin{align*}
\mathcal{H}_1=&\sumstar_{0<d\leq X} \sum_{n\geq  1} \sum_{\substack{n_1,\ldots,n_{m-1}\geq 1 \\ nn_1\dots n_{m-1}=\square}} \chi_d\Bigg(n\prod_{i=1}^{m-1}n_i\Bigg) \lambda_f(n)\Phi\Big(\frac{n}{Y}\Big)\prod_{i=1}^{m-1}\lambda_f(n_i)\Phi\Big(\frac{n_i}{Y^{\varepsilon}}\Big) \\
= & \frac 6{\pi^2}X\sum_{\substack{n,n_1,\dots,n_{m-1} \geq 1\\ nn_1\dots n_{m-1}=\square}}\lambda_f(n)\lambda_f(n_1)\cdots \lambda_f(n_{m-1})\Phi\Big(\frac{n}{Y}\Big)\Phi\Big(\frac{n_1}{Y^{\varepsilon}}\Big)\cdots \Phi\Big(\frac{n_{m-1}}{Y^{\varepsilon}}\Big)\prod_{p|nn_1\cdots n_{m-1}}(1+1/p)^{-1}\\
& +
\begin{cases}
 O(X^{1/2}\displaystyle \sum_{\substack{n \ll Y \\ n_1,\dots,n_{m-1} \ll Y^{\varepsilon}}}|\lambda_f(n)\lambda_f(n_1)\cdots \lambda_f(n_{m-1})|(nn_1\dots n_{m-1})^{1/4+\varepsilon}), \quad
\text{unconditionally} \\
O(X^{1/2+\varepsilon}\displaystyle\sum_{\substack{n \ll Y \\ n_1,\dots,n_{m-1} \ll Y^{\varepsilon}}}|\lambda_f(n)\lambda_f(n_1)\cdots \lambda_f(n_{m-1})|(nn_1\dots n_m)^{\varepsilon}), \quad
\text{under GRH}
\end{cases} \\
\gg & XY^{1/2+(m-1)\varepsilon/2}(\log Y)^{\frac{m(m-3)}{2}} +
\begin{cases}
 O(X^{1/2+\varepsilon}Y^{5/4+\varepsilon}), \quad
\text{unconditionally} \\
O(X^{1/2+\varepsilon}Y^{1+\varepsilon}), \quad
\text{under GRH}
\end{cases}.
\end{align*}

  It follows that we have unconditionally that for $X^{\varepsilon} \ll Y\ll X^{2/3-\alpha}$ for any $\alpha>0$,
\begin{align}
\label{H1bound}
\mathcal{H}_1 \gg XY^{1/2+(m-1)\varepsilon/2}(\log Y)^{\frac{m(m-3)}{2}}\gg XY^{1/2+(m-1)\varepsilon/2}(\log X)^{\frac{m(m-3)}{2}}.
\end{align}
  We also note that the above estimation continues to hold for $X^{\varepsilon} \ll Y\ll X^{1-\alpha}$ under GRH.

 We now apply \eqref{Holder}, \eqref{H2bound} and \eqref{H1bound} to see that for even integers $m \geq 4$, unconditionally for $X^{\varepsilon} \ll Y\ll X^{2/3-\alpha}$ for any $\alpha>0$,
   $$T_m(X,Y;f,\Phi)  \gg XY^{m/2}(\log X)^{\frac{m(m-3)}{2}},$$
while the above estimation continues to hold for $X^{\varepsilon} \ll Y\ll X^{1-\alpha}$ under GRH. This completes the proof of Theorem \ref{lower}.

\vspace*{.5cm}

\noindent{\bf Acknowledgments.}  P. G. is supported in part by NSFC grant 12471003.

\bibliography{biblio}
\bibliographystyle{amsxport}

\end{document}